\documentclass[12pt]{article}
\usepackage{amssymb}
\usepackage{amsmath,amsthm}
\usepackage{cite}
\usepackage[dvips]{graphicx}
\usepackage{hyperref}
\usepackage{color}
\usepackage{setspace}
\usepackage{enumerate}
\usepackage{tkz-graph}
\usepackage{tikz}

\hypersetup{colorlinks=true}

\hypersetup{colorlinks=true, linkcolor=blue, citecolor=blue,urlcolor=blue}

\newtheorem{remark}{Remark}

 \newtheorem{lemma}[remark]{Lemma}
 \newtheorem{theorem}[remark]{Theorem}
 
 \newtheorem{corollary}[remark]{Corollary}

\title{On the Roman domination number of generalized Sierpi\'{n}ski graphs}

\author{F. Ramezani$^{(1)}$, E. D. Rodr\'{i}guez-Bazan$^{(2)}$, J. A. Rodr\'{\i}guez-Vel\'{a}zquez$^{(3)}$ 
\vspace{0,2cm }
\\
$^{(1)}${\small Department of Mathematics, 
Faculty of Science,}\\
{\small  Imam Khomeini International University}  {\small
P. O. Box 34194-288,
Qazvin, Iran.} \\{\small
ramezani\@@ikiu.ac.ir}  
\vspace{0,2cm }
\\
$^{(2)}${\small Department of Matemathics,}\\
{\small Central University of Las Villas,}{ \small   Carretera a Camajuan\'{i} km. $5\frac{1}{2}$. Villa Clara, Cuba.} \\{\small
erickrodriguezbazan\@@gmail.com}
\vspace{0,2cm }
\\
$^{(3)}${\small Departament d'Enginyeria Inform\`atica i Matem\`atiques,}\\
{\small Universitat Rovira i Virgili,}  {\small Av. Pa\"{\i}sos
Catalans 26, 43007 Tarragona, Spain.} \\{\small
 juanalberto.rodriguez\@@urv.cat}
}

\begin{document}
\maketitle

\begin{abstract}
 A map $f : V \rightarrow \{0, 1, 2\}$ is a Roman dominating function on a
graph $G=(V,E)$ if for every vertex $v\in V$ with $f(v) = 0$, there exists a
vertex $u$, adjacent to $v$, such that $f(u) = 2$. The weight of a Roman
dominating function is given by $f(V) =\sum_{u\in V}f(u)$. The
minimum weight of a Roman dominating function on $G$ is called the
Roman domination number of $G$. In this article we study the Roman domination number of Generalized Sierpi\'{n}ski graphs $S(G,t)$. More precisely, we obtain a general upper bound on the Roman domination number of $S(G,t)$ and we discuss the tightness of this bound. In particular, we focus  on the cases in which the base graph $G$ is a  path, a cycle,  a complete graph or a graph having exactly one universal vertex. 
\end{abstract}



\section{Introduction}
Let $G=(V,E)$ be a non-empty graph of order $n\ge 2$, and $t$ a positive integer.  We denote by $V^t $ the set of words of length $t$ on alphabet $V $. The letters of a word $u$ of length $t$ are denoted by $u_1u_2...u_t$. The concatenation of two words $u$ and $v$  is denoted by $uv$. Klav\v{z}ar and Milutinovi\'c introduced in \cite{Klavzar1997} the graph  $S(K_n, t)$, $t\ge 1$,  whose vertex set is $V^t$, where
$\{u,v\}$ is an edge if and only if there exists $i\in \{1,...,t\}$ such that:\\
$$ \mbox{ (i) }  u_j=v_j, \mbox{ if } j<i; \mbox{ (ii) } u_i\ne v_i; \mbox{ (iii) } u_j=v_i \mbox{ and } v_j=u_i  \mbox{ if } j>i.$$
As noted in \cite{Hinz2013},  in a compact form, the edge sets can be described as
$$\{\{wu_iu_j^{r-1},wu_ju_i^{r-1}\}:\,  u_i,u_j\in V, i\ne j; r\in \{1,...,t\}; w\in V^{t-r} \}.$$
The graph $S(K_3,t)$ is isomorphic to 
the graph of the  Tower of Hanoi with $t$ disks \cite{Klavzar1997}. Later, those graphs have been
called Sierpi\'{n}ski graphs in \cite{Klavzar2002} and they were studied by now from numerous points of view. 
For instance,  the authors of \cite{Gravier...Parreau} studied identifying codes, locating-dominating codes, and total-dominating codes in Sierpi\'{n}ski graphs. In \cite{Hirtz-Holz} the authors  propose an algorithm, which makes use of three automata and the fact that there are at most two internally vertex-disjoint shortest paths between any two vertices, to determine all shortest paths in Sierpi\'{n}ski graphs. The authors of \cite{Klavzar2002} proved that for any $n\ge 1$ and $t\ge 1$, the Sierpi\'{n}ski graph $S(K_n,t)$ has a unique 1-perfect code (or efficient dominating set) if $t$ is even, and $S(K_n,t)$ has exactly $n$ 1-perfect codes if $t$ is odd. 
The Hamming dimension of a graph $G$ was introduced in \cite{Klavsar-Peterin} as the largest dimension of a Hamming graph into which $G$ embeds as an irredundant induced subgraph. 
That paper gives an upper bound for the Hamming dimension of the Sierpi\'{n}ski graphs $S(K_n,t)$ for $n\ge 3$. It also showed that the Hamming dimension of $S(K_n,t)$ grows as $3^{t-3}$. 
The idea of almost-extreme vertex of $S(K_n,t)$ was  introduced in \cite{Klavsar-Zeljic} as a
vertex that is either adjacent to an extreme vertex of $S(K_n,t)$
or is incident to
an edge between two subgraphs of $S(K_n,t)$
isomorphic to $S(K_n,t-1)$.  The authors of \cite{Klavsar-Zeljic} deduced explicit formulas  for the distance in $S(K_n,t)$
between an arbitrary vertex and an almost-extreme vertex.
Also they gave  a formula of the metric dimension of a Sierpi\'{n}ski graph, which was independently obtained by Parreau in her Ph.D. thesis.  For a general background  on Sierpi\'{n}ski graph, the reader is
invited to read  the comprehensive survey \cite{Klavzar2016(survey)}  and references
therein.

This construction was generalized in \cite{GeneralizedSierpinski} for any graph $G=(V,E)$, by defining the $t$-th \emph{generalized Sierpi\'{n}ski graph} of $G$, denoted by  $S(G,t)$,  as the graph with vertex set $V^t$ and edge set defined as 
$$\{\{wu_iu_j^{r-1},wu_ju_i^{r-1}\}:\,  \{u_i,u_j\}\in E; r\in  \{1,...,t\}; w\in V^{t-r} \}.$$

\begin{figure}[ht]
\centering
\begin{tikzpicture}[transform shape, inner sep = .5mm]
\def\side{.5};
\pgfmathsetmacro\radius{\side/sqrt(3)};
\foreach \ind in {3,4,5}
{
\pgfmathparse{150-120*(\ind-3)};
\node [draw=black, shape=circle, fill=black] (\ind) at (\pgfmathresult:\radius cm) {};
}
\node [draw=black, shape=circle, fill=black] (1) at ([yshift=\side cm]3) {};
\node [draw=black, shape=circle, fill=black] (2) at ([yshift=\side cm]4) {};
\node [draw=black, shape=circle, fill=black] (6) at ([xshift=-\side cm]5) {};
\node [draw=black, shape=circle, fill=black] (7) at ([shift=({135:\side cm})]6) {};
\foreach \ind in {2,4,5}
{
\node [scale=.8] at ([xshift=.3 cm]\ind) {$\ind$};
}
\foreach \ind in {1,3,6,7}
{
\node [scale=.8] at ([xshift=-.3 cm]\ind) {$\ind$};
}
\foreach \u/\v in {1/2,1/3,2/4,3/5,4/5,5/6,6/7}
{
\draw (\u) -- (\v);
}


\def\widenodetwo{.4};
\pgfmathparse{15*\side};
\node (center) at (\pgfmathresult cm,0) {};
\pgfmathsetmacro\sideTwo{5*\side};
\pgfmathsetmacro\radiuscenter{\sideTwo/sqrt(3)};
\foreach \d in {3,4,5}
{
\pgfmathparse{150-120*(\d-3)};
\node (center\d) at ([shift=({\pgfmathresult:\radiuscenter cm})]center) {};
\foreach \ind in {3,4,5}
{
\pgfmathparse{150-120*(\ind-3)};
\node [draw=black, shape=circle, fill=black,inner sep = \widenodetwo mm] (\d\ind) at ([shift=({\pgfmathresult:\radius cm})]center\d) {};
}
\node [draw=black, shape=circle, fill=black,inner sep = \widenodetwo mm] (\d1) at ([yshift=\side cm]\d3) {};
\node [draw=black, shape=circle, fill=black,inner sep = \widenodetwo mm] (\d2) at ([yshift=\side cm]\d4) {};
\node [draw=black, shape=circle, fill=black,inner sep = \widenodetwo mm] (\d6) at ([xshift=-\side cm]\d5) {};
\node [draw=black, shape=circle, fill=black,inner sep = \widenodetwo mm] (\d7) at ([shift=({135:\side cm})]\d6) {};
\foreach \ind in {2,4,5}
{
\node [scale=.6] at ([xshift=.3 cm]\d\ind) {$\d\ind$};
}
\foreach \ind in {1,3,6,7}
{
\node [scale=.6] at ([xshift=-.3 cm]\d\ind) {$\d\ind$};
}
\foreach \u/\v in {1/2,1/3,2/4,3/5,4/5,5/6,6/7}
{
\draw (\d\u) -- (\d\v);
}
}
\foreach \d in {1,2,6,7}
{
\ifthenelse{\d=1}
{
\node (center1) at ([shift=({0,\sideTwo})]center3) {};
}
{
\ifthenelse{\d=2}
{
\node (center2) at ([shift=({0,\sideTwo})]center4) {};
}
{
\ifthenelse{\d=6}
{
\node (center6) at ([shift=({-\sideTwo,0})]center5) {};
}
{
\node (center7) at ([shift=({135:\sideTwo cm})]center6) {};
};
};
};
\foreach \ind in {3,4,5}
{
\pgfmathparse{150-120*(\ind-3)};
\node [draw=black, shape=circle, fill=black,inner sep = \widenodetwo mm] (\d\ind) at ([shift=({\pgfmathresult:\radius})]center\d) {};
}

\node [draw=black, shape=circle, fill=black,inner sep = \widenodetwo mm] (\d1) at ([yshift=\side cm]\d3) {};
\node [draw=black, shape=circle, fill=black,inner sep = \widenodetwo mm] (\d2) at ([yshift=\side cm]\d4) {};
\node [draw=black, shape=circle, fill=black,inner sep = \widenodetwo mm] (\d6) at ([xshift=-\side cm]\d5) {};
\node [draw=black, shape=circle, fill=black,inner sep = \widenodetwo mm] (\d7) at ([shift=({135:\side cm})]\d6) {};
\foreach \ind in {2,4,5}
{
\node [scale=.6] at ([xshift=.3 cm]\d\ind) {$\d\ind$};
}
\foreach \ind in {1,3,6,7}
{
\node [scale=.6] at ([xshift=-.3 cm]\d\ind) {$\d\ind$};
}
\foreach \u/\v in {1/2,1/3,2/4,3/5,4/5,5/6,6/7}
{
\draw (\d\u) -- (\d\v);
}
}
\foreach \u/\v in {1/2,1/3,2/4,3/5,4/5,5/6,6/7}
{
\ifthenelse{\u=1 \AND \v>2}
{
\draw (\u\v) -- (\v\u);
}
{
\ifthenelse{\u=3 \AND \v=5 \OR \u=6}
{
\draw (\u\v) to[bend right] (\v\u);
}
{
\draw (\u\v) to[bend left] (\v\u);
};
};
}
\end{tikzpicture}
\caption{A graph $G$ and the Sierpi\'{n}ski graph $S(G,2)$.}
\label{FigS(G,2).}
\end{figure}
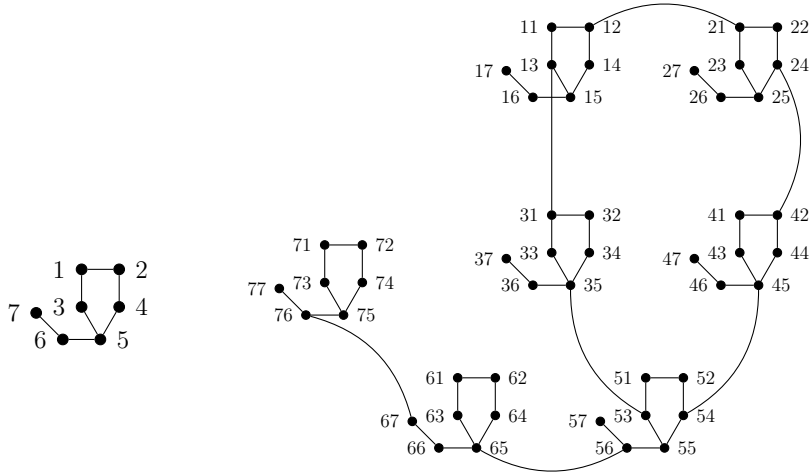

Figure \ref{FigS(G,2).} shows a graph $G$ and the generalized Sierpi\'{n}ski graph $S(G,2)$, while Figure \ref{FigS(G,3).} shows the Sierpi\'{n}ski graph $S(G,3)$.

Notice that if $\{u,v\}$ is an edge of $S(G,t)$, there is an edge $\{x,y\}$ of $G$ and a word $w$ such that $u=wxyy\dots y$ and $v=wyxx\dots x$. In general, $S(G,t)$ can be constructed recursively from $G$ with the following process: $S(G,1)=G$ and, for $t\ge 2$, we copy $n$ times $S(G, t-1)$ and add the letter $x$ at the beginning of each label of the vertices belonging to  the copy of $S(G,t-1)$ corresponding to $x$. Then for every edge $\{x,y\}$ of $G$, add an edge between vertex $xyy\dots y$ and vertex $yxx\dots x$. See, for instance, Figure \ref{FigS(G,3).}. Vertices of the form $xx\dots x$ are called \textit{extreme vertices} of $S(G,t)$. Notice that for any graph $G$ of order $n$ and any integer $t\ge 2$,  $S(G,t)$  has $n$ extreme vertices and, if $x$ has degree $d(x)$ in $G$, then the extreme vertex $xx\dots x$ of $S(G,t)$   also has degree  $d(x)$. Moreover,   the degrees of two vertices $yxx\dots x$ and  $xyy\dots y$, which connect two copies of $S(G,t-1)$, are  equal to  $d(x)+1$ and $d(y)+1$, respectively.

For any $w\in V^{t-1}$ and $t\ge 2$  the subgraph $\langle V_w \rangle$ of $S(G,t)$, induced by $V_w=\{wx:\; x\in V\}$, is isomorphic to $G$. Notice that there exists only one vertex $u\in V_w$ of the form $w'xx\ldots x$, where $w'\in V^{r}$ for some $r\le t-2$. We will say that $w'xx\ldots x$ is \textit{the extreme vertex} of $\langle V_w \rangle$, which is an extreme vertex in $S(G,t)$ whenever $r=0$. By definition of $S(G,t)$ we deduce the following remark.

\begin{remark}
Let $G=(V,E)$ be a graph, let $t\ge 2$ be an integer and $w\in V^{t-1}$.  If  $u\in V_w$ and $v\in V^t-V_w$ are adjacent in $S(G,t)$, then   either $u$ is the extreme vertex of  $\langle V_w \rangle$ or $u$ is adjacent to the extreme vertex  of  $\langle V_w \rangle$.
\end{remark}

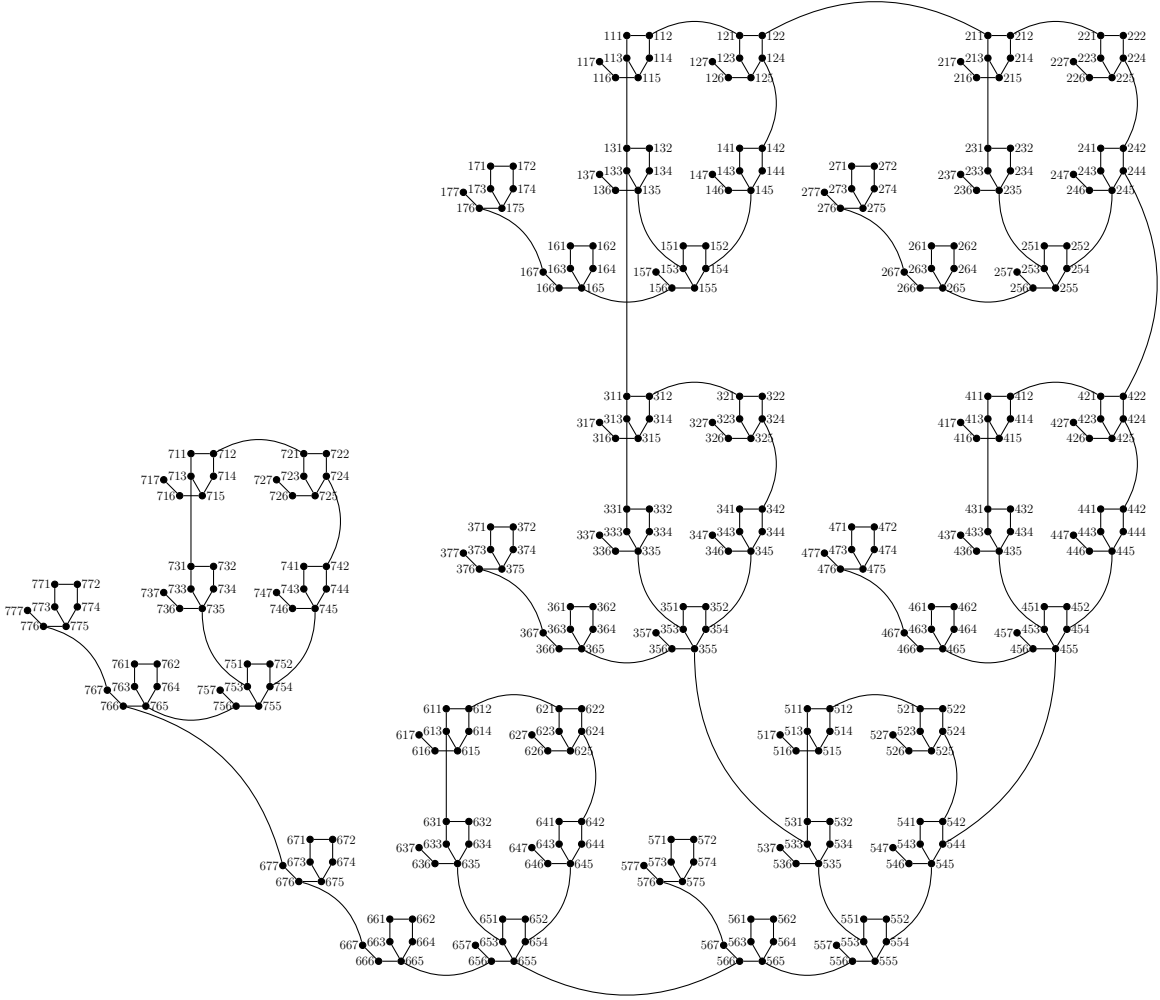
\begin{figure}[ht]
\begin{tikzpicture}[transform shape, inner sep = .3mm]


\def\side{.3};
\pgfmathsetmacro\radius{\side/sqrt(3)};
\pgfmathsetmacro\sideTwo{5*\side};
\pgfmathsetmacro\radiuscenter{\sideTwo/sqrt(3)};
\pgfmathsetmacro\sideThree{16*\side};
\pgfmathsetmacro\radiuscenterThree{\sideThree/sqrt(3)};
\node (center) at (0,0) {};
\foreach \c in {3,4,5}
{
\pgfmathparse{150-120*(\c-3)};
\node (center\c) at ([shift=({\pgfmathresult:\radiuscenterThree cm})]center) {};


\foreach \d in {3,4,5}
{
\pgfmathparse{150-120*(\d-3)};
\node (center\c\d) at ([shift=({\pgfmathresult:\radiuscenter cm})]center\c) {};
\foreach \ind in {3,4,5}
{
\pgfmathparse{150-120*(\ind-3)};
\node [draw=black, shape=circle, fill=black] (\c\d\ind) at ([shift=({\pgfmathresult:\radius cm})]center\c\d) {};
}
\node [draw=black, shape=circle, fill=black] (\c\d1) at ([yshift=\side cm]\c\d3) {};
\node [draw=black, shape=circle, fill=black] (\c\d2) at ([yshift=\side cm]\c\d4) {};
\node [draw=black, shape=circle, fill=black] (\c\d6) at ([xshift=-\side cm]\c\d5) {};
\node [draw=black, shape=circle, fill=black] (\c\d7) at ([shift=({135:\side cm})]\c\d6) {};
\foreach \ind in {2,4,5}
{
\node [scale=.4] at ([xshift=.18 cm]\c\d\ind) {$\c\d\ind$};
}
\foreach \ind in {1,3,6,7}
{
\node [scale=.4] at ([xshift=-.18 cm]\c\d\ind) {$\c\d\ind$};
}
\foreach \u/\v in {1/2,1/3,2/4,3/5,4/5,5/6,6/7}
{
\draw (\c\d\u) -- (\c\d\v);
}
}
\foreach \d in {1,2,6,7}
{
\ifthenelse{\d=1}
{
\node (center\c1) at ([shift=({0,\sideTwo})]center\c3) {};
}
{
\ifthenelse{\d=2}
{
\node (center\c2) at ([shift=({0,\sideTwo})]center\c4) {};
}
{
\ifthenelse{\d=6}
{
\node (center\c6) at ([shift=({-\sideTwo,0})]center\c5) {};
}
{
\node (center\c7) at ([shift=({135:\sideTwo cm})]center\c6) {};
};
};
};
\foreach \ind in {3,4,5}
{
\pgfmathparse{150-120*(\ind-3)};
\node [draw=black, shape=circle, fill=black] (\c\d\ind) at ([shift=({\pgfmathresult:\radius})]center\c\d) {};
}

\node [draw=black, shape=circle, fill=black] (\c\d1) at ([yshift=\side cm]\c\d3) {};
\node [draw=black, shape=circle, fill=black] (\c\d2) at ([yshift=\side cm]\c\d4) {};
\node [draw=black, shape=circle, fill=black] (\c\d6) at ([xshift=-\side cm]\c\d5) {};
\node [draw=black, shape=circle, fill=black] (\c\d7) at ([shift=({135:\side cm})]\c\d6) {};
\foreach \ind in {2,4,5}
{
\node [scale=.4] at ([xshift=.18 cm]\c\d\ind) {$\c\d\ind$};
}
\foreach \ind in {1,3,6,7}
{
\node [scale=.4] at ([xshift=-.18 cm]\c\d\ind) {$\c\d\ind$};
}
\foreach \u/\v in {1/2,1/3,2/4,3/5,4/5,5/6,6/7}
{
\draw (\c\d\u) -- (\c\d\v);
}
}
\foreach \u/\v in {1/2,1/3,2/4,3/5,4/5,5/6,6/7}
{
\ifthenelse{\u=1 \AND \v>2}
{
\draw (\c\u\v) -- (\c\v\u);
}
{
\ifthenelse{\u=3 \AND \v=5 \OR \u=6}
{
\draw (\c\u\v) to[bend right] (\c\v\u);
}
{
\draw (\c\u\v) to[bend left] (\c\v\u);
};
};
}
}
\foreach \c in {1,2,6,7}
{
\ifthenelse{\c=1}
{
\node (center1) at ([shift=({0,\sideThree})]center3) {};
}
{
\ifthenelse{\c=2}
{
\node (center2) at ([shift=({0,\sideThree})]center4) {};
}
{
\ifthenelse{\c=6}
{
\node (center6) at ([shift=({-\sideThree,0})]center5) {};
}
{
\node (center7) at ([shift=({135:\sideThree cm})]center6) {};
};
};
};


\foreach \d in {3,4,5}
{
\pgfmathparse{150-120*(\d-3)};
\node (center\c\d) at ([shift=({\pgfmathresult:\radiuscenter cm})]center\c) {};
\foreach \ind in {3,4,5}
{
\pgfmathparse{150-120*(\ind-3)};
\node [draw=black, shape=circle, fill=black] (\c\d\ind) at ([shift=({\pgfmathresult:\radius cm})]center\c\d) {};
}
\node [draw=black, shape=circle, fill=black] (\c\d1) at ([yshift=\side cm]\c\d3) {};
\node [draw=black, shape=circle, fill=black] (\c\d2) at ([yshift=\side cm]\c\d4) {};
\node [draw=black, shape=circle, fill=black] (\c\d6) at ([xshift=-\side cm]\c\d5) {};
\node [draw=black, shape=circle, fill=black] (\c\d7) at ([shift=({135:\side cm})]\c\d6) {};
\foreach \ind in {2,4,5}
{
\node [scale=.4] at ([xshift=.18 cm]\c\d\ind) {$\c\d\ind$};
}
\foreach \ind in {1,3,6,7}
{
\node [scale=.4] at ([xshift=-.18 cm]\c\d\ind) {$\c\d\ind$};
}
\foreach \u/\v in {1/2,1/3,2/4,3/5,4/5,5/6,6/7}
{
\draw (\c\d\u) -- (\c\d\v);
}
}
\foreach \d in {1,2,6,7}
{
\ifthenelse{\d=1}
{
\node (center\c1) at ([shift=({0,\sideTwo})]center\c3) {};
}
{
\ifthenelse{\d=2}
{
\node (center\c2) at ([shift=({0,\sideTwo})]center\c4) {};
}
{
\ifthenelse{\d=6}
{
\node (center\c6) at ([shift=({-\sideTwo,0})]center\c5) {};
}
{
\node (center\c7) at ([shift=({135:\sideTwo cm})]center\c6) {};
};
};
};
\foreach \ind in {3,4,5}
{
\pgfmathparse{150-120*(\ind-3)};
\node [draw=black, shape=circle, fill=black] (\c\d\ind) at ([shift=({\pgfmathresult:\radius})]center\c\d) {};
}

\node [draw=black, shape=circle, fill=black] (\c\d1) at ([yshift=\side cm]\c\d3) {};
\node [draw=black, shape=circle, fill=black] (\c\d2) at ([yshift=\side cm]\c\d4) {};
\node [draw=black, shape=circle, fill=black] (\c\d6) at ([xshift=-\side cm]\c\d5) {};
\node [draw=black, shape=circle, fill=black] (\c\d7) at ([shift=({135:\side cm})]\c\d6) {};
\foreach \ind in {2,4,5}
{
\node [scale=.4] at ([xshift=.18 cm]\c\d\ind) {$\c\d\ind$};
}
\foreach \ind in {1,3,6,7}
{
\node [scale=.4] at ([xshift=-.18 cm]\c\d\ind) {$\c\d\ind$};
}
\foreach \u/\v in {1/2,1/3,2/4,3/5,4/5,5/6,6/7}
{
\draw (\c\d\u) -- (\c\d\v);
}
}
\foreach \u/\v in {1/2,1/3,2/4,3/5,4/5,5/6,6/7}
{
\ifthenelse{\u=1 \AND \v>2}
{
\draw (\c\u\v) -- (\c\v\u);
}
{
\ifthenelse{\u=3 \AND \v=5 \OR \u=6}
{
\draw (\c\u\v) to[bend right] (\c\v\u);
}
{
\draw (\c\u\v) to[bend left] (\c\v\u);
};
};
}
}
\foreach \u/\v in {1/2,1/3,2/4,3/5,4/5,5/6,6/7}
{
\ifthenelse{\u=1 \AND \v>2}
{
\draw (\u\v\v) -- (\v\u\u);
}
{
\ifthenelse{\u=3 \AND \v=5 \OR \u=6}
{
\draw (\u\v\v) to[bend right] (\v\u\u);
}
{
\draw (\u\v\v) to[bend left] (\v\u\u);
};
};
}
\end{tikzpicture}
\caption{The Sierpi\'{n}ski graph $S(G,3)$ for the graph $G$ of Figure \ref{FigS(G,2).}. }
\label{FigS(G,3).}
\end{figure}

The authors of  \cite{GeneralizedSierpinski} announced some results about generalized Sierpi\'{n}ski graphs concerning their automorphism groups  and perfect codes. These results definitely deserve
to be published. Since then some papers have been published on various aspects of 
generalized Sierpi\'{n}ski graphs. For instance, in \cite{Rodriguez-Velazquez2015a} their chromatic number,
vertex cover number, clique number, and domination number, are investigated.
The authors of \cite{Rodriguez-Velazquez2015} obtained closed formulae for the Randi\'{c} index of polymeric networks modelled by  generalized Sierpi\'{n}ski graphs, while in \cite{2015arXiv151007982E}  this work was extended to the so-called generalized Randi\'{c} index. Also, the total chromatic number of generalized Sierpi\'{n}ski graphs was  studied  in \cite{Geetha2015} and the strong metric dimension has recently been  studied in \cite{Rodriguez-Velazquez2016}.
In this paper we obtain closed formulae or bounds on the Roman domination number of generalized Sierpi\'{n}ski graphs $S(G,t)$ in terms of parameters of the base graph $G$. 

We begin by establishing the principal terminology and notation
which we will use throughout the article. Hereafter $G=(V,E)$
denotes a finite simple graph of order $n\ge 2$. The distance between two vertices $x,y\in V$ will be denoted by $d_G(x,y)$. For two adjacent vertices $u$ and $v$ of $G$ we use the notation  $u\sim v$. For a  vertex
$v$ of $G$, $N(v)=\{u\in V:\; u\sim v\}$ denotes the set of neighbors that $v$ has in $G$. $N(v)$ is called the \emph{open neighborhood of} $v$
 and the \emph{close neighborhood of} $v$ is defined as $N[v]=N(v)\cup \{v\}$.
For a set $D\subseteq V$, the \emph{open neighborhood }is $N(D)=\cup_{v\in D}N(v)$ and the \emph{closed neighborhood} is $N[D]=N(D)\cup D$. A set  $D$ is a {\em dominating set} if $N[D]=V$.
 The \emph{domination number} $\gamma(G)$  is the minimum cardinality of a dominating set in $G$.  We say that a set $S$ is a $\gamma(G)$-set if it is a dominating set and $|S|=\gamma(G)$.
The subgraph induced a subset $S$ of vertices will be denoted by $\langle S\rangle$.

A map $f : V \rightarrow \{0, 1, 2\}$ is a \emph{Roman dominating function} on a
graph $G$ if for every vertex $v$ with $f(v) = 0$, there exists a
vertex $u\in N(v)$ such that $f(u) = 2$. The {\em weight} of a Roman
dominating function is given by $f(V) =\sum_{u\in V}f(u)$. The
minimum weight of a Roman dominating function on $G$ is called the
{\em Roman domination number} of $G$ and it is denoted by
$\gamma_{_R}(G)$.

Any Roman dominating function $f$ on a graph $G$ induces three sets $B_0,
B_1, B_2$, where $B_i = \{v\in V\;:\; f(v) = i\}$. Thus, we will write $f=(B_0,B_1,B_2)$.  It is clear that for any
Roman dominating function $f=(B_0,B_1,B_2)$ on a graph $G=(V,E)$ of order $n$ we
have that $f(V)=\sum_{u\in V}f(u)=2|B_2|+|B_1|$ and
$|B_0|+|B_1|+|B_2|=n$.  We say that a function $f=(B_0,B_1,B_2)$ is a $\gamma_{R}(G)$-function if it is a Roman dominating function and $f(V)=\gamma_{_R}(G)$.

The Roman domination number was introduced by Cockayne et al. \cite{Cockayne2004} in 2004 and  since then about 100 papers have been published on various aspects of Roman domination in graphs (for examples, see \cite{Bermudo-Roman,Chellali-Roman2016}). For instance, in \cite{Cockayne2004,Henning2003} was obtained the following result, which shows the relationship between the domination number and the Roman domination number of a graph. 

\begin{lemma}{\em \cite{Cockayne2004,Henning2003}}\label{lema-roman-dom}
For any graph $G$, $\gamma(G)\le \gamma_{_R}(G)\le 2\gamma(G)$.
\end{lemma}

As shown in \cite{Cockayne2004}, $\gamma(G)= \gamma_{_R}(G)$ if and only if $G$ is an empty graph.  A graph $G$ is said to be a \emph{Roman graph} if $\gamma_{_R}(G)=2\gamma(G)$. Several examples of Roman graphs are given in 
\cite{Cockayne2004,Xueliang2009,YeroJA2013}.
\begin{theorem}{\em \cite{Cockayne2004}}\label{lema-RomanGraph}
A graph $G$ is Roman if and only if it has a $\gamma_{_R}$-function $f=(B_0,\emptyset,B_2)$.
\end{theorem}

The following result, stated in \cite{Cockayne2004}, will be used as a tool to study the Roman domination number of $S(G,t)$ for the cases in which the base graph is a path or a cycle.

\begin{theorem} {\rm \cite{Cockayne2004}} \label{RomanDominationPn} For the classes of paths $P_n$ and cycles $C_n$, $\gamma_{_R}(P_n)=\gamma_{_R}(C_n)=\lceil \frac{2n}{3} \rceil$.
\end{theorem}


Let $G=(V,E)$ be a graph, and $H=(V,E')$ a subgraph of $G$. Since any $\gamma_{_R}(H)$-function is a Roman dominating function of $G$, we can state the following remark.

\begin{remark} \label{RomanSubgraph}
Let $G=(V,E)$ be a graph, and $H=(V,E')$ a subgraph of $G$.
 Then $\gamma_{_R}(G)\le \gamma_{_R}(H)$. 
\end{remark}

\section{An upper bound on the Roman domination number}

%
%

Let $f=(B_0,B_1,B_2) $  be a $\gamma_{_R}$-function on $G$   and let $D_i$ be the set of non-isolated vertices of $\langle B_i \rangle$ for $i \in \lbrace 0,1,2 \rbrace$. Also, let $D_{12}$   be the set of non-isolated vertices of $\langle B_1 \cup B_2\rangle $. Notice that, if we take $f$ such that $|B_1|$ is minimum, then $B_1$ is an independent set, which implies that $D_1=\emptyset$ and $D_{1,2}=D_2$. 
With these notations in mind  we state the following result.

\begin{theorem}\label{UpBoundRomanDomSierpinski} Let $G$ be a graph of order $n$. For any $\gamma_{_R}$-function $f=(B_0,B_1,B_2) $   on $G$, and any integer $t \geq 2$,
\begin{center}
$\gamma_{_R}(S(G,t)) \leq n^{t-2}( n \gamma_{_R}(G)-|B_2|-|D_{12}|-\theta +\frac{1}{2} |D_1|),$
\end{center}
where $\theta=|\lbrace u\in B_1\setminus D_1: \, d_G(u,v)=2 \text{ for some } v \in B_2 
\text{ such that }
 |N_G(v) \cap B_0|=2   \rbrace   |$.
\end{theorem}

\begin{proof}
 Let $f=(B_0,B_1,B_2)$ be a $\gamma_{_R}(G)$-function. For a given integer $t \geq 2 $ we define $S_i=\lbrace wx ; \quad w \in V^{t-1} \: , \: x \in B_i \rbrace$, for $i \in \lbrace 0,1,2 \rbrace$. Let $g:V^t \rightarrow \lbrace0,1,2\rbrace$ such that $g=(S_0,S_1,S_2)$. If $v \in V^t$ and $g(v)=0$, then $v=wy$ where $w$ is a word in $V^{t-1}$ and $y \in B_0$. Since $f$ is a $\gamma_{_R}$-function on $G$, there is $z \in B_2 \cap N_G(y)$. Hence $wz \in S_2 \cap N_{S(G,t)} (wy)$. So $g$ is a Roman dominating function on $S(G,t)$ and $\gamma_{_R}(S(G,t)) \leq \omega(g)=n^{t-1}(|B_1|+2|B_2|)=n^{t-1} \gamma_{_R}(G)$. Now we have four steps for reach the result.\\
 \\
{ Step 1}: Set $S'_2 = \lbrace wuu: \, w \in V^{t-2},  \; u \in B_2\rbrace$. We define $g_1: V^t \rightarrow \lbrace0,1,2\rbrace$ such that $g_1=(S_0,S_1 \cup S'_2,S_2 \setminus S'_2)$. Let $y \in S_0$. Then $y$ has the form $wuv_0$ where $w \in V^{t-2}$, $v_0 \in B_0$ and $u\in V$. Since $f$ is a $\gamma_{_R}(G)$-function, there is $v_2 \in B_2$ such that $v_0 $ is adjacent to $v_2$ in graph $ G$. So $wuv_0$ is adjacent to $wuv_2$.  If $wuv_2 \in S_2\setminus S'_2$, then we are done. Now, if $wuv_2\in S'_2$, then 
 $v_2=u$ and, since $v_0$ is adjacent to $v_2$,  we can conclude that  
$y=wv_2v_0$ is adjacent to $wv_0v_2$. Hence $g_1$ is a Roman dominating function on $S(G,t)$. Therefore $\gamma_{_R}(S(G,t)) \leq \omega(g_1)=n^{t-2}(n\gamma_{_R}(G)-|B_2|)$.\\
\\
{Step 2}: Set $S''_2=\lbrace wvv: \,  w \in V^{t-2} ,\,  v \in D_2 \rbrace$. We define $g_2:V^t \rightarrow \lbrace0,1,2\rbrace$ where
\begin{eqnarray*}g_2(x )= \left \{ \begin{array}{ll} 

0, & x \in S''_2;
\\
g_1(x) , & \, \text{otherwise}.
\end{array} \right
.\end{eqnarray*}
Let  $x \in V^t$ such that $g_2(x)=0 $. In this case, $g_1(x)=0$ or  $x\in S''_2$.

Suppose that  $g_1(x)=0$. Since $x$ must belong to $S_0$,  it is of the form $x=wuv_0$, where $w \in V^{t-2}           
$, $u\in V$ and $v_0\in B_0$.  If $N_{S(G,t)}(x)\cap S_2''=\emptyset$, then there exists $y\in N_{S(G,t)}(x)\cap (S_2\setminus S_2')$. On the other side, if  $z\in N_{S(G,t)}(x)\cap S_2''$, then $z=wv_2v_2$, where $v_2\in D_2$ and $u=v_2$, and so $v_2\sim v_0$,  which implies that $x=wv_2v_0\sim wv_0v_2$, and we know that $g_2(wv_0v_2)=g_1(wv_0v_2)=g(wv_0v_2)=2.$

Now, if $x\in S''_2$, then there exists $w\in V^{t-2}$ and $v \in D_2$ such that $x=wvv$. So, by definition of $D_2$,  $x$ must be  adjacent to $wvu$ for some  $u \in D_2 \setminus \lbrace v \rbrace$. Hence $g_2(wvu)=g_1(wvu)=g(wvu)=f(u)=2$. 

Therefore, $g_2$ is a Roman dominating function on $S(G,t)$, and so
 $\gamma_{_R}(S(G,t)) \leq n^{t-2}(n\gamma_{_R}(G)-|B_2|-|D_2|)$.\\
\\
{Step 3}: We know that the maximum degree on $\langle B_1\rangle $ is one. Since $ D_1$ is the set of the non-isolated vertices of $\langle B_1\rangle $,  $\langle D_1 \rangle \cong \cup_{i=1}^{k} P_2$, where $k = \frac{1}{2} |D_1|$. Suppose that  $\lbrace v_1,u_1,v_2,u_2,\ldots , v_k,u_k\rbrace$ is the vertex set of 
 $\langle D_1 \rangle$, where   $ v_i\sim u_i $  for $1 \leq i \leq k$. Set 
 $S'_1=\lbrace wv_iv_i : \,  w \in V^{t-2},\,   v_i \in D_1 \text{ and } 1\leq i \leq k\rbrace$, $S''_1=\lbrace wu_iv_i :\,  w \in V^{t-2}, \, v_i,u_i \in D_1   \text{ and }  1\leq i \leq k\rbrace$ and $S'''_1=\lbrace wv_iu_i: \, wu_iv_i \in S_1'' \rbrace$. We define $g_3:V^t\rightarrow \lbrace0,1,2\rbrace$ such that 
\begin{eqnarray*}g_3(x )= \left \{ \begin{array}{ll} 

0, & x \in S'_1 \cup S''_1;
\\
2 ,& x \in S'''_1;
\\
g_2(x), & \, otherwise.
\end{array} \right
.\end{eqnarray*}
Notice that $S'''_1$ is a dominating set for $S'_1 \cup S''_2$. So $g_3$ is a Roman dominating function on $S(G,t)$. Also $\omega(g_3)=\omega (g_2)-|S'_1|-|S''_1|+|S'''_1|$ and $|S'_1|=\frac{n^{t-2}}{2} |D_1|$. Hence $\gamma_{_R}(S(G,t)) \leq n^{t-2}(n\gamma_{_R}(G) -|B_2|-|D_2|-\frac{1}{2} |D_1|)$. \\
We know that there are not any edges between $B_1$ and $B_2$. So $|D_{12}|=|D_1|+|D_2|$. Hence $\gamma_{_R}(S(G,t)) \leq n^{t-2}( n \gamma_{_R}(G)-|B_2|-|D_{12}|+\frac{1}{2} |D_1| )$.\\
\\
{Step 4}: Let $B_2'=\lbrace v \in B_2:  |N_G(v) \cap B_0|=2 \text{ and } d_G(v,u)=2 \text{ for some } u \in B_1 \setminus D_1 \rbrace$. Let $\Pi$ be the set of paths  $v_0, w_2, w_0,w_1$ in  $G$ such that  $w_2 \in B_2'$, $v_0,w_0 \in B_0$ and $w_1 \in B_1\setminus D_1$. Given two vertices $x,y\in V$,  we say that $\mu(x,y)=(i,j)$ if there exist  a path $v_0, w_2, w_0,w_1$ in $\Pi$ such that $x$ and $y$ are (from the left) in position $i$ and $j$, respectively.   
We define the following sets.
\vspace{-0,7cm}
\begin{verse}
\item $A_1=\lbrace wxy :   w \in V^{t-2} \text{ and } \mu(x,y)=(3,4) \rbrace$, 
\item $A_2=\lbrace wxy :   w \in V^{t-2} \text{ and } \mu(x,y)=(4,4) \rbrace$,
\item $A_3=\lbrace wxy :   w \in V^{t-2} \text{ and } \mu(x,y)=(4,2) \rbrace$,
\item $A_4=\lbrace wxy :   w \in V^{t-2} \text{ and } \mu(x,y)=(4,1) \rbrace$,
\item $A_5=\lbrace wxy :   w \in V^{t-2} \text{ and } \mu(x,y)=(4,3) \rbrace$.
\end{verse}

Notice that $|A_2|=\theta$ and  $A_i\cap A_j = \emptyset $, for all $i\ne j$, $i,j\in \{1,\dots,5\}$.
Also, since the weight of $f$ is minimum, for every $w_2\in B_2'$ there exists exactly one vertex $w_1\in B_1\setminus D_1$ such that $d_G(w_2,w_1)=2$. Hence, $|A_3|=|B_2'|$. Furthermore, since $|N_G(w_2)\cap B_0|=2$, we can conclude that 
$|A_1|=|A_4|=|A_5|$. On the other hand, suppose that there are two different paths 
$v_0,w_2,w_0,w_1$ and $v_0,w_2',w_0',w_1$ in $\Pi$. In such a case, the weight of the cycle $v_0,w_2,w_0,w_1,w_0',w_2',v_0 $ equals $5$ and we can give a new $\gamma_{_R}(G)$-function where the weight is $4$, as we can consider  that $v_0$ and $w_1$ have label  $2$ and the remaining vertices have label $0$, which is a contradiction with the minimality of $f$. Hence, $|A_4|=|B_2'|$. 
Now, define the function  $g_4:V^t \rightarrow \lbrace0,1,2\rbrace$ such that\\
\begin{eqnarray*}g_4(v )= \left \{ \begin{array}{ll} 

0, & v \in A_1 \cup A_2 \cup A_3 ;
\\
1, & v \in A_4;
\\
2, & v \in A_5;
\\
g_3(v), & \, \text{otherwise}.
\end{array} \right
.\end{eqnarray*}

\begin{figure}[h]
\centering
\begin{tikzpicture}
[line width=1pt,scale=1]
\draw[black!40](6,2)--(8,2);
\draw[black!40](8,2)--(10,2);
\draw[black!40](8,2)--(8,1);
\draw[black!40](8,2)--(8,1);
\draw[black!40](8,0)--(8,1);
\coordinate [label=right:{$ww_1w_2$}] (e) at (8,1);
\coordinate [label=right:{$ww_0w_1$}] (f) at (5.3,2.3);
\coordinate [label=right:{$ww_1w_1$}] (h) at (9.4,2.3);
\coordinate [label=right:{$ww_1v_0$}] (h) at (8,0);
\coordinate [label=right:{$ww_1w_0$}] (h) at (7.4,2.3);
\coordinate [label=right:{$0$}] (e) at (7.4,1);
\coordinate [label=right:{$0$}] (f) at (5.7,1.7);
\coordinate [label=right:{$0$}] (h) at (9.7,1.7);
\coordinate [label=right:{$1$}] (h) at (7.4,0.0);
\coordinate [label=right:{$2$}] (h) at (7.4,1.7);
\filldraw[fill opacity=0.9,fill=black]  (6,2) circle (0.07cm);
\filldraw[fill opacity=0.9,fill=black]  (8,1) circle (0.07cm);
\filldraw[fill opacity=0.9,fill=black]  (8,2) circle (0.07cm);
\filldraw[fill opacity=0.9,fill=black]  (10,2) circle (0.07cm);
\filldraw[fill opacity=0.9,fill=black]  (8,0) circle (0.07cm);
[line width=1pt,scale=1]
\draw[black!40](-2,2)--(0,2);
\draw[black!40](0,2)--(2,2);
\draw[black!40](0,2)--(0,1);
\draw[black!40](0,2)--(0,1);
\draw[black!40](0,0)--(0,1);
\coordinate [label=right:{$ww_1w_2$}] (e) at (0,1);
\coordinate [label=right:{$ww_0w_1$}] (f) at (-2.7,2.3);
\coordinate [label=right:{$ww_1w_1$}] (h) at (1.4,2.3);
\coordinate [label=right:{$ww_1v_0$}] (h) at (0,0);
\coordinate [label=right:{$ww_1w_0$}] (h) at (-0.6,2.3);
\coordinate [label=right:{$2$}] (e) at (-0.6,1);
\coordinate [label=right:{$1$}] (f) at (-2.3,1.7);
\coordinate [label=right:{$1$}] (h) at (1.7,1.7);
\coordinate [label=right:{$0$}] (h) at (-0.6,0.0);
\coordinate [label=right:{$0$}] (h) at (-0.6,1.7);
\filldraw[fill opacity=0.9,fill=black]  (-2,2) circle (0.07cm);
\filldraw[fill opacity=0.9,fill=black]  (0,1) circle (0.07cm);
\filldraw[fill opacity=0.9,fill=black]  (0,2) circle (0.07cm);
\filldraw[fill opacity=0.9,fill=black]  (2,2) circle (0.07cm);
\filldraw[fill opacity=0.9,fill=black]  (0,0) circle (0.07cm);
\coordinate [label=right:{$\longrightarrow$}] (e) at (4,1);
\end{tikzpicture}
\caption{This figure shows how the labels imposed by function $g_3$ are transformed by function $g_4$.} \label{Step 4}
\end{figure}
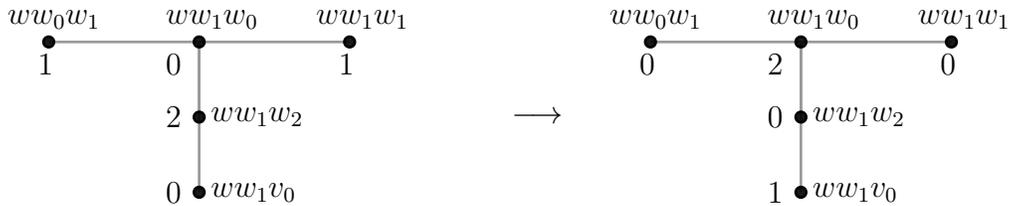

Notice that $ A_5$ is a dominating set for $A_1 \cup A_2 \cup A_3 $. So $g_4$ is a Roman dominating function on $S(G,t)$ (See Figure \ref{Step 4}). 
 Then 
\begin{align*}
 \omega(g_4) &= 2|A_5|+|A_4|+\omega(g_3)-|A_1|-|A_2|-2|A_3|\\
                     &=\omega(g_3)-\theta \\
                     &\le n^{t-2}\left ( n \gamma_{_R}(G)-|B_2|-|D_{12}|-\theta +\frac{|D_1|}{2} \right),
\end{align*} 
 as required. 
\end{proof}

As we can see in Theorems \ref{RomanSierpinskiDegree(n-1)} and  \ref{RomanSierpinskiPaths} the bound above is achieved for any graph having exactly one universal vertex and for                                                                                                                                                                                                                                                                                    any path $P_n$, where  $n\equiv 0,1  \pmod 3$.

Since  any Roman graph has a $\gamma_{_R}$-function $f=(B_0,\emptyset,B_2)$, we can state the following particular case of Theorem \ref{UpBoundRomanDomSierpinski}.
\begin{corollary}\label{UpBoundRomanDomSierpinskiRomanGraphs} For any    Roman graph $G$ of order $n$ and any integer $t \geq 2$,
$$\gamma_{_R}(S(G,t)) \leq \gamma(G) n^{t-2}(2n-1).$$
\end{corollary}

\section{Graphs having exactly one universal vertex}

\begin{theorem}\label{RomanSierpinskiDegree(n-1)}
If $G$ is a graph of order $n\ge 4$ having exactly one vertex of degree $n-1$, then for any integer $t \geq 2$, $\gamma_{_R}(S(G,t))=n^{t-2}(2n-1).$ 
\end{theorem}
\begin{proof}
By Theorem \ref{UpBoundRomanDomSierpinski} we deduce that $\gamma_{R}(S(G,t)) \leq  n^{t-2}(2n-1).$ We will show that for any $\gamma_{_R}$-function
$f=(B_0,B_1,B_2)$ on $S(G,t)$, $\omega (f) \ge n^{t-2}(2n-1).$ Let $V=\lbrace 0,1,\ldots, n-1 \rbrace$ such that $deg(0)=n-1.$
We would point out  that for any $w\in V^{t-2}$, $i\in V$ and $t\ge 3$,  the subgraph $\langle V_{wi} \rangle$ of $S(G,t)$, induced by $V_{wi}=\{wij:\; j\in V\}$, is isomorphic to $G$. 
Let $\lambda (V_{wi}) = |\lbrace wij \in V_{wi}: \, deg(wij) \neq deg(j ) \rbrace |.$ There are two general cases.\\
{Case I.}   $i\ne 0.$ In this case $1 \leq \lambda(V_{wi}) \leq n-1$. So there is $wij \in V_{wi}$ such that $deg(wij)=deg(j)$ for $1 \leq j \leq n-1.$ If $B_2 \cap V_{wi} \neq \emptyset$, then $\omega(V_{wi}) \geq 2.$ Otherwise, $wij \in B_1$ and $\omega(V_{wi}) \geq 1.$ If $\omega(V_{wi})=1$, then $f(wik)=0$ for $k \in V \setminus \lbrace j \rbrace.$ Let $l \in V \setminus \lbrace 0,i,j \rbrace$. Then $wil \in N(wli)$ where $wli \in B_2.$ Since $l\ne 0$, $\lambda(V_{wl}) \leq n-1$, and so there is $wll' \in V_{wl} \cap (B_1 \cup B_2)$ such that $l' \neq i.$ Hence $\omega(V_{wl}) \geq 3$. This shows that $\omega(V_{wi})+\omega(V_{wl}) \geq 4.$ Therefore, for every copy of $G$ of weight $1$ there is another copy of $G$ of weight at least $3.$ Since there are $n^{t-2}(n-1) $ copies of $G$ of this type in $S(G,t)$, the contribution of these copies of $G$ to $\omega (f)$ equals $\displaystyle\sum_{w \in V^{t-2}}  \sum_{i=1}^{n-1} \omega (V_{wi})\geq 2n^{t-2}(n-1).$ \\
{ Case II.} $i=0.$ Then $n-1 \leq \lambda(V_{w0}) \leq n.$ If $V_{w0} \nsubseteq B_0$, then $\omega (V_{w0}) \geq 1.$ Suppose that $\omega(V_{w0})=0.$ Hence $ \lambda(V_{w0})=n$ and so $w00 \in N(w'jj)$ for $w' \in V^{t-2}$ and $j\ne 0.$ Since $f$ is a $\gamma_R$-function, $w'jj \in B_2.$ Also $deg(j) < n-1$, so there is $z\in V\setminus \{0,j\}$  such that $w'jz \not \in N(w'jj)$ and $deg(w'jz)= deg(z).$ Hence, $f(w'jz)\in \{  0,1\}$. If $f(w'jz)=1$, then we can move  the label $2$ from $w'jj$ to $w'00$ and the label $1$ from $w'jz$ to $w00$. The function obtained in this manner is a $\gamma_{_R}$-function on  $S(G,t)$, and so we can assume that $f$ is defined in this manner, \textit{i.e.,} $\omega(V_{w0})=1$. Now, 
If $f(w'jz)=0$, then we have two possibilities. Either $f(w'j0)=2$ or  $f(w'jl)=2$, for some $l\in N(z)$. The case  $f(w'j0)=2$ is impossible, as we can put the label $1$ to $w00$ and the label $0$ to $w'jj$, and the function obtained is a Roman dominating function of weight less than $f$, which is a contradiction. Finally,  if $f(w'jl)=2$, then we can modify the following weights: we put label $2$ to $w'j0$, label $0$ to $w'jl$ label $1$ to $w00$, label $0$ to $w'jj$ and, if $l\in N(j)$, then we put label  $1$ to $w'lj$. The function obtained in this manner is a $\gamma_{_R}$-function on  $S(G,t)$, and so we can assume that $f$ is defined in this manner, \emph{i.e.,} $\omega(V_{w0})=1$.
 So $\displaystyle \sum_{w \in V^{t-2}} \omega (V_{w0})\geq n^{t-2}.$ 
 
Therefore,  $\gamma_{_R}(S(G,t))=\omega (f) \geq n^{t-2}+2n^{t-2}(n-1)=n^{t-2}(2n-1).$ The proof is completed. 
\end{proof}

Since any graph of order $n$  having  at most one vertex of degree greater than or equal to $n-2$ is the subgraph of a graph of order $n$ having exactly one vertex of degree $n-1$, Remark \ref{RomanSubgraph} and Theorem \ref{RomanSierpinskiDegree(n-1)} lead to the following result. 

\begin{theorem}
If $G$ is a graph of order $n\ge 4$  having  at most one vertex of degree greater than or equal to $n-2$, then for any integer $t \geq 2$, $\gamma_{_R}(S(G,t))\ge n^{t-2}(2n-1).$ 
\end{theorem}

\section{The particular case of paths}
Notice that $S(P_2,t)\cong P_{2^t}$ and so $\gamma_{_R}(S(P_2,t))=\left\lceil \frac{2^{t+1}}{3} \right\rceil$.  From now on we assume that $n\ge 3$. Let  $V=\lbrace 1,2,\ldots,n \rbrace$ be the vertex set of $P_n$, and $\langle V_{wu} \rangle$ a copy of $P_n$ in $S(P_n,t)$  for $w \in V^{t-2}$ and $u \in V$.
Set 
\begin{eqnarray*}A_{wu}= \left \{ \begin{array}{ll} 
 \lbrace wui \in V_{wu} : \, i< u-1 \rbrace, & 3 \leq u \leq n \; ;\\
\\
\emptyset, &   u = 1, 2.

\end{array} \right
.\end{eqnarray*}
\begin{eqnarray*}B_{wu}= \left \{ \begin{array}{ll} 
 \lbrace wuj \in V_{wu} : \, j > u+1 \rbrace, & 1 \leq u \leq n-2 \; ;\\
\\
\emptyset, &  u = n- 1, n.

\end{array} \right
.\end{eqnarray*}\\
Also, let
 $$D_i= \left \{\langle V_{wu} \rangle : \, \omega(V_{wu})=\left \lceil \frac{2|A_{wu}|}{3} \right \rceil +\left \lceil \frac{2|B_{wu}|}{3} \right \rceil +i \right\},\, \text{ for  }i\in \{0,1\}$$ 
 and 
 $$D_2= \left \{\langle V_{wu} \rangle : \, \omega(V_{wu})=\left \lceil \frac{2|A_{wu}|}{3} \right \rceil +\left \lceil \frac{2|B_{wu}|}{3} \right \rceil +j, \text{ for some } j\ge 2 \right\},$$
 where the weight $\omega(V_{wu})$ corresponds  to a labelling defined by a $\gamma_{_R}$-function on $S(P_n,t)$.
Also set $\Lambda= \{ \langle V_{wu} \rangle: \, deg(wuu)\neq deg(u) \text{ for }   1 \leq u \leq n \}$. 
With these notations in mind we will prove the following Lemmas.

\begin{lemma}\label{weightV_wu} Let $f=(B_0,B_1,B_2)$ be a $\gamma_{_R}$-function on $S(P_n,t)$, where $n\ge 3$. For any   $w \in V^{t-2}$ and $u \in V$ there exists $i \geq 0$ such that $\langle V_{wu} \rangle \in D_i$, and $i \geq 1$ whenever $V_{wu} \notin \Lambda$.
\end{lemma}
\begin{proof}
Let $P_r=\langle A_{wu}\rangle$ and $P_{r'}=\langle B_{wu}\rangle$. Notice that Theorem \ref{RomanDominationPn} leads to $\gamma_{_R}(\langle A_{wu}\rangle)=\lceil \frac{2r}{3}\rceil$  and $\gamma_{_R}(\langle B_{wu}\rangle)=\lceil \frac{2r'}{3}\rceil$. If $V_{wu} \not\in \Lambda$, then $deg(wuu)=deg(u) \leq 2$. 
 Since $$\omega (V_{wu})=\displaystyle \omega ( A_{wu})+ \sum_{wui\not \in A_{wu} \cup B_{wu}}f(wui)+\omega ( B_{wu}),$$ $\omega (V_{wu})\geq \omega ( A_{wu})+\omega ( B_{wu})+1.$
 If $ \omega ( A_{wu})\geq \lceil \frac{2r}{3}\rceil $ or $ \omega ( B_{wu})\geq \lceil \frac{2r'}{3}\rceil $, then we are done.
If $  A_{wu} \neq \emptyset$ and $ \omega ( A_{wu})< \lceil \frac{2r}{3}\rceil $, then 
 $f(wu(u-2))=0$, $f(wu(u-3))\le 1$, and so $f(wu(u-1))=2$. Hence,  $ \omega ( A_{wu})+f(wu(u-1))= \lceil \frac{2(r-2)}{3}\rceil +1+2\ge  \lceil \frac{2r}{3}\rceil +1$.
 By analogy, if  $  B_{wu} \neq \emptyset$ and $ \omega ( B_{wu})<\lceil \frac{2r'}{3}\rceil $, then  $ \omega ( B_{wu})+f(wu(u+1))\ge  \lceil \frac{2r'}{3}\rceil +1$. Therefore, in any case,
 $$\omega (V_{wu})\ge \left \lceil \frac{2|A_{wu}|}{3} \right \rceil +\left \lceil \frac{2|B_{wu}|}{3} \right \rceil+1.$$
 Let $V_{wu} \in \Lambda$. Then $wuu \in N(w'vv)$ where $w' \in V^{t-2}$ and $v \in V.$ Thus, as above,
 $$\omega (V_{wu})=\displaystyle   \omega ( A_{wu})+ \sum_{wui\not \in A_{wu} \cup B_{wu}}f(wui)+ \omega ( B_{wu})\geq \left \lceil \frac{2|A_{wu}|}{3} \right \rceil +\left \lceil \frac{2|B_{wu}|}{3} \right \rceil.$$
\end{proof}

\begin{lemma}\label{D2>D0}  
 Let $V$ be the vertex set of $P_n$, $n\ge 3$, and $t$ a positive integer.
If for   some $w\in V^{t-2}$ and $u\in V$ we have that $\langle V_{wu} \rangle \in D_0$, then there exists  $w'\in V^{t-2}$ and  $v\in N_G(u)$ such that $\langle V_{w'v} \rangle \in D_2.$
\end{lemma}
\begin{proof} Let $f=(B_0,B_1,B_2)$ be $\gamma_{_R}$-function on $S(P_n,t)$, and  $\langle V_{wu} \rangle \in D_0$. Then $\displaystyle \sum_{wui\not \in A_{wu} \cup B_{wu}}f(wui)=0.$ Thus, $wuu \in N(w'vv)$ where $ w'vv \in V^{t-2} \cap B_2$ for $w' \in V^{t-2}$ and $ v \in V$. Hence, $\langle V_{w'v} \rangle \in \Lambda$ and $\omega(V_{w'v}) \geq \left \lceil \frac{2|A_{w'v}|}{3} \right \rceil +\left \lceil \frac{2|B_{w'v}|}{3} \right \rceil +2.$ So, $\langle V_{w'v} \rangle \in D_2.$ 

\end{proof}

\begin{theorem}\label{RomanSierpinskiPaths}
 For any  integers $n\ge 3$ and $t \geq 2$,\\
\begin{eqnarray*}\gamma_{_R}(S(P_n,t))= \left \{ \begin{array}{ll} 
n^{t-2}\left (n \lceil \frac{2n}{3}\rceil  -\lceil \frac{n}{3}\rceil \right), & n\equiv 0,1 \; \pmod 3;\\
\\
n^{t-2}\left (n \lceil \frac{2n}{3}\rceil  -2 \lceil \frac{n}{3}\rceil+1 \right),  &  n\equiv  2 \; \pmod 3.

\end{array} \right
.\end{eqnarray*}\\
\end{theorem}

\begin{proof}
We first proceed to deduce the lower bound $\gamma_{_R}(S(P_n,t))\ge n^{t-2}(n \lceil \frac{2n}{3}\rceil  -\lceil \frac{n}{3}\rceil )$.
Let  $V=\lbrace 1,2,\ldots,n \rbrace$, and  $f=(B_0,B_1,B_2)$ a $\gamma_{_R}$-function on $S(P_n,t)$. Let  $\langle V_{wu} \rangle$ be a copy of $P_n$ in $S(P_n,t)$  for $w \in V^{t-2}$ and $u \in V$. 
Since $$\gamma_{_R}(S(P_n,t))=\displaystyle \sum_{w\in V^{t-2},\,u\in V }\omega(V_{wu}),$$ we proceed to obtain a lower bounds on $\omega(V_{wu})$ in terms of $n$. 
Before doing it, notice that 
 $$
  \gamma_{_R}(S(P_n,t)) 
 =   \sum_{\langle V_{wu}\rangle\in D_0 }\omega(V_{wu})+\sum_{\langle V_{wu}\rangle\in D_1 }\omega(V_{wu})+\sum_{\langle V_{wu}\rangle\in D_ 2 }\omega(V_{wu}) 
$$
and  by Lemma \ref{D2>D0},  there exists an injective application $\psi: D_0 \longrightarrow D_2$, so that we emphasize that if $\langle V_{wu}\rangle \in D_0$, then  the contribution of $\omega(V_{wu})+\omega(\psi(\langle V_{wu}\rangle))$ to  $\gamma_{_R}(S(P_n,t)) $
is greater than or equal to its contribution when 
both $\langle V_{wu} \rangle$ and $\psi(\langle V_{wu}\rangle )$ are in $D_1$.  With this observation in mind we continue the proof.

By Lemma \ref{weightV_wu}, $\omega( V_{wu})= \left \lceil \frac{2|A_{wu}|}{3} \right \rceil +\left \lceil \frac{2|B_{wu}|}{3} \right \rceil+i $, for some $i\ge 0$. 
Hence, we now proceed to express $\left \lceil \frac{|A_{wu}|}{3} \right \rceil $ and $\left \lceil \frac{2|B_{wu}|}{3} \right \rceil $ in terms of $n$. To this end,  we consider the set $S=\lbrace x \in V : \, x \equiv 2 \, (\text{mod} \, 3 \, ) \rbrace$ and we differentiate three cases. \\
\\
\noindent Case 1: $n=3k$ for some positive integer $k$. So $S$ is a $\gamma-$set of $P_n$. 
If  $u \in S$, then $|A_{wu}|,|B_{wu}| \in \lbrace 3k' : \, 0 \leq k' \leq k-1\rbrace $ and, as  $|A_{wu}\cup B_{wu}|=n-3$, we have 
\begin{equation}\label{1}
 \omega (V_{wu})=2\frac{n-3}{3}+i= \frac{2n}{3} +i-2.
\end{equation} 
If $u \in N(S)\setminus \{1,n\}$, then  $|A_{wu}| \in \lbrace l : \, l \equiv 1(\text{mod} \, 3 ) \rbrace $ and $|B_{wu}| \in \lbrace l : \, l \equiv 2(\text{mod} \, 3 ) \rbrace $ or vice versa. 
Hence, $\omega(V_{wu}) = \frac{2n}{3}+i-1$. Notice that if $u=1$, then $A_{wu}=\emptyset $ and $|B_{wu}|\equiv 1 \pmod 3$, which implies that  $\omega(V_{wu}) = \frac{2n}{3}+i-1$. The case $u=n$ is analogous to the previous one.  
Therefore, 
 \begin{align*}
  \gamma_{_R}(S(P_n,t)) &=\sum_{w\in V^{t-2}}  \sum_{u\in V}\omega(V_{wu})\\
   &\geq n^{t-2} \left(\left(\frac{2n}{3}-1\right)\gamma(P_n)+ \frac{2n}{3}\left(n - \gamma(P_n) \right)\right)\\
   &=n^{t-2}\left( n \left\lceil \frac{2n}{3}\right\rceil  -\left\lceil \frac{n}{3}\right\rceil \right).
   \end{align*}

\noindent Case 2:  $n= 3k+1$ for any positive integer $k$. In this case, $S'=S \cup \lbrace n-1 \rbrace$ is a $\gamma$-set of $P_n$. If $\langle V_{wd}\rangle $ is a copy of $P_n$ for some  $d \in S'$,
 $|A_{wd}| \in \lbrace l : \, l \equiv 0 \pmod 3 \rbrace $ and $|B_{wd}| \in \lbrace l : \, l \equiv 1 \pmod 3 \rbrace $ or vice versa.
 Hence, 
 \begin{equation}\label{2} 
 \omega(V_{wd}) =  \left \lceil \frac{2|A_{wd}|}{3}\right \rceil +\left \lceil \frac{2|B_{wd}|}{3}\right \rceil +i = 2\left\lfloor \frac{n}{3}\right\rfloor+i-1.
\end{equation}

Let $V_{wu}$ where $u \in N(S') \setminus \lbrace 1,n \rbrace$. Hence, we have two possibilities,   $|A_{wu}|,|B_{wu}| \in \lbrace l: \, l \equiv 2   \pmod  3  \rbrace$ or 
$|A_{wu}|,|B_{wu}| \in  \lbrace l: \, l \equiv 0,1 \pmod 3 \rbrace$ where $|A_{wu}| \not \equiv |B_{wu}| \pmod 3$. In the first case, 
  $\omega(V_{wu}) \geq 2\lfloor \frac{n}{3} \rfloor +i$ and, in the second one,  $\omega(V_{wu}) \geq 2\lfloor \frac{n}{3} \rfloor +i-1$.

  Suppose that $\omega (V_{wv})=2\lfloor \frac{n}{3} \rfloor+i-1$  for $w \in V^{t-2}$ and $v \in V$. Then $\omega(V_{w(v-1)})> 2\lfloor \frac{n}{3} \rfloor+i-1$ where $v-1 \in S$. Therefore $\omega(V_{wu})$ is equal to $2\lfloor \frac{n}{3} \rfloor+i-1$ at most for $\gamma(P_n)$ copies of $P_n$, and for other copies it is more than $ 2\lfloor \frac{n}{3} \rfloor+i-1$. Hence, $$\gamma_{_R}(S(P_n,t))\geq n^{t-2} \left (2\gamma(P_n) \left \lfloor \frac{n}{3} \right \rfloor + (n-\gamma(P_n))\left(2\left \lfloor \frac{n}{3} \right \rfloor+1\right)\right)=n^{t-2}\left (n \left \lceil \frac{2n}{3}\right\rceil  -\left\lceil \frac{n}{3}\right\rceil \right).$$ 

\noindent Case 3:  $n=3k+2$ for any positive integer $k$. 
We discuss first words of the form $ wu$ where $2\leq u \leq  n-1$ and $w \in V^{t-2}$.  If $wuu \in B_2 \cup B_1$, then $\omega (V_{wu}) \geq  \lceil \frac{2(u-2)}{3}\rceil + \lceil \frac{2(n-u-1)}{3}\rceil+1$. Hence $\omega (V_{wu}) \geq 2\lfloor \frac{n}{3} \rfloor+1 $ for $u\equiv 0 \, (\text{mod} \, 3)$ and  $\omega (V_{wu}) \geq 2\lfloor \frac{n}{3} \rfloor+2 $ for others.  
Now,  suppose that $wuu \in B_0$ and $\langle V_{wu}\rangle\notin D_0$. In this case $wu(u-1) \in B_2$ or $wu(u+1) \in B_2$, say  $wu(u+1) \in B_2$. Hence, $\omega (V_{wu}) \geq \lceil \frac{2(u-2)}{3}\rceil + \lceil \frac{2(n-u-2)}{3}\rceil +2$, which implies that  $\omega (V_{wu}) \geq 2\lfloor \frac{n}{3} \rfloor+1 $ for $u \in \lbrace3k',3k'+2: \, 0\leq k' \leq k-1 \rbrace$ and $\omega (V_{w'}) \geq 2\lfloor \frac{n}{3} \rfloor+2 $ for others. In summary, 
$\omega (V_{wu}) \geq 2\lfloor \frac{n}{3} \rfloor+1 $ for $u\equiv 0,2 \pmod 3$ and $\omega (V_{wu}) \geq 2\lfloor \frac{n}{3} \rfloor+2$ for $u\equiv 1 \pmod 3$.

Now, let  $ u \in  \lbrace 1,n \rbrace$. Suppose that $u=1$ (for $u=n$, the proof is likewise). If $\langle V_{w1}\rangle \in D_2$, then $\omega (V_{w1}) \geq 2\lfloor \frac{n}{3} \rfloor+2$. Now, if $\langle V_{w1}\rangle \in D_1$, then $f(w11)=1$ or $f(w11)=0$.
In the first case, $f(w21)=2$, as $f(w13)=2$ implies that $\langle V_{w1}\rangle \in D_2$, which is a contradiction. In the second case, there exists $w' \in V^{t-2}$ such that $f(w'22)=2$ and  $w11 \in N(w'22) $.  As a consequence, $\omega(V_{w1})\geq 2\lfloor \frac{n}{3} \rfloor+2$ or for some $w' \in V^{t-2}$, $\omega(V_{w'2})\geq 2\lfloor \frac{n}{3} \rfloor+2$. In summary, we can collect the lower bounds for weight of the copies of $P_n$ in $S(P_n,2)$ in a table.

\begin{center}
\begin{tabular}{|l|l|l|l|l|}
\hline
  & $u=$ & $3k'$ & $3k'+1$ & $3k'+2$\\
 \hline
 &$u\ne 1,n$,  $ \omega (V_{wu}) \ge $ & $2\lfloor \frac{n}{3} \rfloor+1$ & $2\lfloor \frac{n}{3} \rfloor+2$ & $2\lfloor \frac{n}{3} \rfloor+1$\\
\hline
$\langle V_{w1}\rangle \in D_0 $  & $\begin{array}{ll}
& \omega (V_{w1}) \ge  \\
& \omega (V_{w2}) \ge \\
& \text{ and } \\
\exists w'\in V^{t-2}: & \omega (V_{w'2}) \ge
\end{array} $ 
 &  & $\begin{array}{l}
  2\lfloor \frac{n}{3} \rfloor   \\
  \\
  \\
\end{array} $   & $\begin{array}{l}
\\
 2\lfloor \frac{n}{3} \rfloor +2  \\
 \\
2\lfloor \frac{n}{3} \rfloor +2  \\
\end{array} $  \\
\hline
$\langle V_{w1}\rangle \in D_1 $  & $\begin{array}{ll}
& \omega (V_{w1}) \ge  \\
& \omega (V_{w2}) \ge \\
& \text{ or } \\
\exists w'\in V^{t-2}: & \omega (V_{w'2}) \ge
\end{array} $ 
 &  & $\begin{array}{l}
  2\lfloor \frac{n}{3} \rfloor +1  \\
  \\
  \\
\end{array} $   & $\begin{array}{l}
\\
 2\lfloor \frac{n}{3} \rfloor +2  \\
 \\
2\lfloor \frac{n}{3} \rfloor +2  \\
\end{array} $  \\
\hline
$\langle V_{w1}\rangle \in D_2 $  &\hspace{2,67cm}  $ \omega (V_{w1}) \ge  $
 & &  $2\lfloor \frac{n}{3} \rfloor +2 $ &  \\
\hline
\end{tabular}
\end{center}
 Therefore, 
 \begin{align*}
 \gamma_{_R}(S(P_n,t))& = \sum_{w\in V^{t-2}}  \sum_{u\in V}\omega(V_{wu})\\
 &\geq n^{t-2 }\left( \left(2 \left\lfloor \frac{n}{3} \right\rfloor +1 \right)\left(2\left\lfloor \frac{n}{3} \right\rfloor+1\right)+\left\lceil \frac{n}{3} \right\rceil \left(2\left\lfloor \frac{n}{3} \right\rfloor+2\right) \right) \\
 &= n^{t-2}\left(n \left\lceil \frac{2n}{3}\right\rceil  -2 \left\lceil \frac{n}{3}\right\rceil+1\right).
 \end{align*}
 and the proof of the lower bound is complete.


For $n\equiv 0,1\,  (\text{mod}\, 3)$, the upper bound $\gamma_{_R}(S(P_n,t))\le n^{t-2}(n \lceil \frac{2n}{3}\rceil  -\lceil \frac{n}{3}\rceil )$ is obtained from Theorem \ref{UpBoundRomanDomSierpinski}. Thus we consider the case $n=3k+2$ for some positive integer $k$. 
 
As above, consider the set $S=\lbrace x \in V \setminus \{n\} : \, x \equiv 2 \, (\text{mod} \, 3 \, ) \rbrace$. In order to construct a  Roman dominating function we introduce the following sets.
\begin{align*}
  & A_1=\lbrace  wis : \, w \in V^{t-2}, \, s \in S, \, i \geq s+2 \rbrace,\\
  & A_2=\lbrace wi(n-1): \, w \in V^{t-2}, \, i\in \lbrace1,n\rbrace \rbrace,\\
 &  A_3=\lbrace wij : \, w \in V^{t-2}, \, 1 \leq i \leq n-2, \, j=i+1+3k', \, 0 \leq k' \leq k-1 \rbrace,\\
 & C_1=\lbrace win : \, w \in V^{t-2}, \, i \in S \rbrace,\\
 & C_2=\lbrace w(s+1)(s-1): \, w \in V^{t-2}, \, s \in S \rbrace,\\
  &C_3=\lbrace w(n-1)(n-1): \, w \in V^{t-2}\rbrace.
\end{align*}
Define $g:V^t \rightarrow \lbrace 0,1,2 \rbrace$ such that
\begin{eqnarray*}g(wij)= \left \{ \begin{array}{ll} 
2, & wij \in \displaystyle \bigcup_{i=1}^{3}A_i;\\

1, & wij \in \displaystyle \bigcup_{i=1}^{3}C_i;\\

0,  & \text{otherwise}.

\end{array} \right.
\end{eqnarray*}\\

\begin{figure}[h]
\centering
\begin{tabular}{cccccc}

\begin{tikzpicture}
[vertex_style/.style={circle,inner sep=0pt,minimum size=0.17cm, fill opacity=0.9, fill=black},
edge_style/.style={line width=1pt, gray},line width=1pt]

\draw[gray](0.75,0.75)--(0.75,1.5);

\draw[gray](2.25,0)--(2.25,0.75);
\draw[gray](1.5,0)--(6.75,0);
\draw[gray](0.75,0.75)--(6,0.75);
\draw[gray](0,1.5)--(5.25,1.5);
\draw[gray](2.25,-0.75)--(7.5,-0.75);
\draw[gray](3.75,-0.75)--(3.75,0);
\draw[gray](5.25,-0.75)--(5.25,-1.5);
\draw[gray](3,-1.5)--(8.25,-1.5);
\draw[gray](3.75,-2.25)--(9,-2.25);
\draw[gray](6.75,-1.5)--(6.75,-2.25);
\draw[gray](4.5,-3)--(9.75,-3);
\draw[gray](8.25,-3)--(8.25,-2.25);
\draw[gray](5.25,-3.75)--(10.5,-3.75);
\draw[gray](9.75,-3.75)--(9.75,-3);

\filldraw[fill opacity=0.9,fill=black]  (0,1.5) circle (0.07cm);
\filldraw[fill opacity=0.9,fill=black]  (0.75,1.5) circle (0.07cm);
\filldraw[fill opacity=0.9,fill=black]  (1.5,1.5) circle (0.07cm);
\filldraw[fill opacity=0.9,fill=black]  (2.25,1.5) circle (0.07cm);
\filldraw[fill opacity=0.9,fill=black]  (3,1.5) circle (0.07cm);
\filldraw[fill opacity=0.9,fill=black]  (3.75,1.5) circle (0.07cm);
\filldraw[fill opacity=0.9,fill=black]  (4.5,1.5) circle (0.07cm);
\filldraw[fill opacity=0.9,fill=black]  (5.25,1.5) circle (0.07cm);

\filldraw[fill opacity=0.9,fill=black]  (0.75,0.75) circle (0.07cm);
\filldraw[fill opacity=0.9,fill=black]  (1.5,0.75) circle (0.07cm);
\filldraw[fill opacity=0.9,fill=black]  (2.25,0.75) circle (0.07cm);
\filldraw[fill opacity=0.9,fill=black]  (3,0.75) circle (0.07cm);
\filldraw[fill opacity=0.9,fill=black]  (3.75,0.75) circle (0.07cm);
\filldraw[fill opacity=0.9,fill=black]  (4.5,0.75) circle (0.07cm);
\filldraw[fill opacity=0.9,fill=black]  (5.25,0.75) circle (0.07cm);
\filldraw[fill opacity=0.9,fill=black]  (6,0.75) circle (0.07cm);

\filldraw[fill opacity=0.9,fill=black]  (1.5,0) circle (0.07cm);
\filldraw[fill opacity=0.9,fill=black]  (2.25,0) circle (0.07cm);
\filldraw[fill opacity=0.9,fill=black]  (3,0) circle (0.07cm);
\filldraw[fill opacity=0.9,fill=black]  (3.75,0) circle (0.07cm);
\filldraw[fill opacity=0.9,fill=black]  (4.5,0) circle (0.07cm);
\filldraw[fill opacity=0.9,fill=black]  (5.25,0) circle (0.07cm);
\filldraw[fill opacity=0.9,fill=black]  (6,0) circle (0.07cm);
\filldraw[fill opacity=0.9,fill=black]  (6.75,0) circle (0.07cm);

\filldraw[fill opacity=0.9,fill=black]  (2.25,-0.75) circle (0.07cm);
\filldraw[fill opacity=0.9,fill=black]  (3,-0.75) circle (0.07cm);
\filldraw[fill opacity=0.9,fill=black]  (3.75,-0.75) circle (0.07cm);
\filldraw[fill opacity=0.9,fill=black]  (4.5,-0.75) circle (0.07cm);
\filldraw[fill opacity=0.9,fill=black]  (5.25,-0.75) circle (0.07cm);
\filldraw[fill opacity=0.9,fill=black]  (6,-0.75) circle (0.07cm);
\filldraw[fill opacity=0.9,fill=black]  (7.5,-0.75) circle (0.07cm);
\filldraw[fill opacity=0.9,fill=black]  (6.75,-0.75) circle (0.07cm);

\filldraw[fill opacity=0.9,fill=black]  (3,-1.5) circle (0.07cm);
\filldraw[fill opacity=0.9,fill=black]  (3.75,-1.5) circle (0.07cm);
\filldraw[fill opacity=0.9,fill=black]  (4.5,-1.5) circle (0.07cm);
\filldraw[fill opacity=0.9,fill=black]  (5.25,-1.5) circle (0.07cm);
\filldraw[fill opacity=0.9,fill=black]  (6,-1.5) circle (0.07cm);
\filldraw[fill opacity=0.9,fill=black]  (8.25,-1.5) circle (0.07cm);
\filldraw[fill opacity=0.9,fill=black]  (7.5,-1.5) circle (0.07cm);
\filldraw[fill opacity=0.9,fill=black]  (6.75,-1.5) circle (0.07cm);

\filldraw[fill opacity=0.9,fill=black]  (3.75,-2.25) circle (0.07cm);
\filldraw[fill opacity=0.9,fill=black]  (4.5,-2.25) circle (0.07cm);
\filldraw[fill opacity=0.9,fill=black]  (5.25,-2.25) circle (0.07cm);
\filldraw[fill opacity=0.9,fill=black]  (6,-2.25) circle (0.07cm);
\filldraw[fill opacity=0.9,fill=black]  (8.25,-2.25) circle (0.07cm);
\filldraw[fill opacity=0.9,fill=black]  (7.5,-2.25) circle (0.07cm);
\filldraw[fill opacity=0.9,fill=black]  (6.75,-2.25) circle (0.07cm);
\filldraw[fill opacity=0.9,fill=black]  (9,-2.25) circle (0.07cm);

\filldraw[fill opacity=0.9,fill=black]  (5.25,-3) circle (0.07cm);
\filldraw[fill opacity=0.9,fill=black]  (6,-3) circle (0.07cm);
\filldraw[fill opacity=0.9,fill=black]  (8.25,-3) circle (0.07cm);
\filldraw[fill opacity=0.9,fill=black]  (7.5,-3) circle (0.07cm);
\filldraw[fill opacity=0.9,fill=black]  (6.75,-3) circle (0.07cm);
\filldraw[fill opacity=0.9,fill=black]  (9,-3) circle (0.07cm);
\filldraw[fill opacity=0.9,fill=black]  (9.75,-3) circle (0.07cm);
\filldraw[fill opacity=0.9,fill=black]  (4.5,-3) circle (0.07cm);

\filldraw[fill opacity=0.9,fill=black]  (6,-3.75) circle (0.07cm);
\filldraw[fill opacity=0.9,fill=black]  (8.25,-3.75) circle (0.07cm);
\filldraw[fill opacity=0.9,fill=black]  (7.5,-3.75) circle (0.07cm);
\filldraw[fill opacity=0.9,fill=black]  (6.75,-3.75) circle (0.07cm);
\filldraw[fill opacity=0.9,fill=black]  (9,-3.75) circle (0.07cm);
\filldraw[fill opacity=0.9,fill=black]  (9.75,-3.75) circle (0.07cm);
\filldraw[fill opacity=0.9,fill=black]  (10.5,-3.75) circle (0.07cm);
\filldraw[fill opacity=0.9,fill=black]  (5.25,-3.75) circle (0.07cm);

\coordinate [label=right:{$2$}] (a) at (0.5,1.75);
\coordinate [label=right:{$2$}] (a) at (4.25,1.75);
\coordinate [label=right:{$2$}] (a) at (2.75,1.75);

\coordinate [label=right:{$1$}] (a) at (5.75,1);
\coordinate [label=right:{$2$}] (a) at (2,1);
\coordinate [label=right:{$2$}] (a) at (4.25,1);

\coordinate [label=right:{$2$}] (a) at (5.75,0.25);
\coordinate [label=right:{$2$}] (a) at (3.5,0.25);
\coordinate [label=right:{$1$}] (a) at (1.25,0.25);

\coordinate [label=right:{$2$}] (a) at (2.75,-0.5);
\coordinate [label=right:{$2$}] (a) at (5,-0.5);
\coordinate [label=right:{$2$}] (a) at (7.25,-0.5);

\coordinate [label=right:{$2$}] (a) at (3.5,-1.25);
\coordinate [label=right:{$2$}] (a) at (6.53,-1.25);
\coordinate [label=right:{$1$}] (a) at (7.92,-1.25);

\coordinate [label=right:{$2$}] (a) at (4.25,-2);
\coordinate [label=right:{$1$}] (a) at (5.75,-2);
\coordinate [label=right:{$2$}] (a) at (8,-2);

\coordinate [label=right:{$2$}] (a) at (5,-2.75);
\coordinate [label=right:{$2$}] (a) at (7.25,-2.75);
\coordinate [label=right:{$1$}] (a) at (8.8,-2.75);

\coordinate [label=right:{$2$}] (a) at (5.8,-3.5);
\coordinate [label=right:{$2$}] (a) at (8,-3.5);
\coordinate [label=right:{$2$}] (a) at (9.65,-3.5);

\draw[gray] (0,1.5) ..  controls  (-0.75,1.5) .. (-1.5,-2);

\filldraw[fill opacity=0.9,fill=black]  (-1.5,-2) circle (0.07cm);
\coordinate [label=right:{$w3122$}] (a) at (-1.6,-1.85);


\draw[gray] (1.5,0.75) ..  controls  (0.6,0.73) .. (-1,-4);
\filldraw[fill opacity=0.9,fill=black]  (-1,-4) circle (0.07cm);
\coordinate [label=right:{$w2333$}] (a) at (-1,-4);

\draw[gray] (3,0) ..  controls  (1.75,-0.25) .. (0.75,-5);
\filldraw[fill opacity=0.9,fill=black]  (0.75,-5) circle (0.07cm);
\coordinate [label=right:{$w3322$}] (a) at (0.8,-5);

\end{tikzpicture} 
\end{tabular}\caption{This figure shows 
the labelling of $\langle V_{w32}\rangle \cong S(P_8,2)$
induced by $g$, where labels $0$'s are omitted. }
\end{figure}
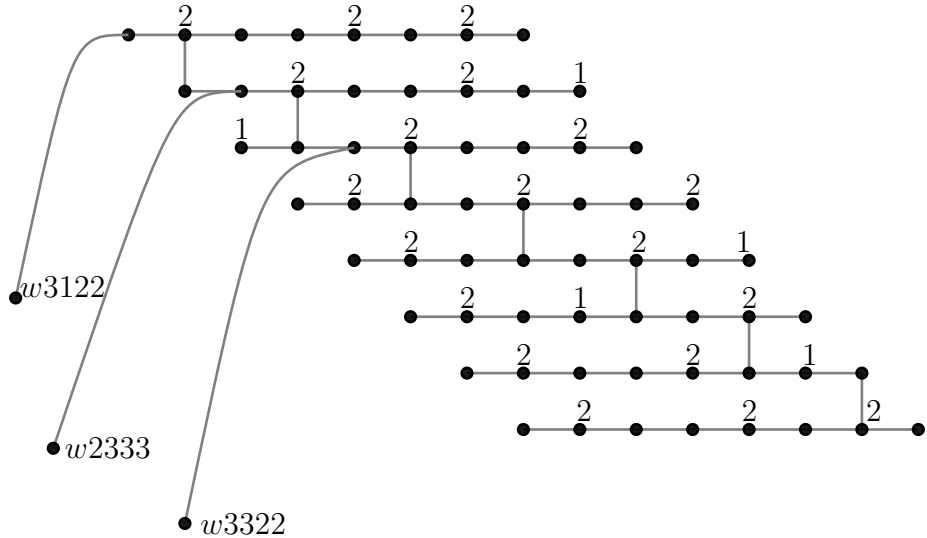

Suppose that $g(wij)=0$ for $w \in V^{t-2}$ and $i,j \in V$. If $i >j+2$, then $j \not\in S$ and so $wij \in N(wis)$ where $s \in \{ j-1,j+1 \}$. As consequence, $i > s+2$ and $wis \in A_1$. If $i=j+2$ and $s=j+1$, then $wij=w(s+1)(s-1) \in C_2$, which is a contradiction, as $g(wij)=0$. Hence, if $i=j+2$, then $s=j-1$ and $wis=w(s+3)s \in A_1.$  Now, let $i <j+2$. If $i=j+1$, then $wij=wi(i-1) \in N(w(i-1)i ) $ and $w(i-1)i \in A_2 $. Also, if $i < j+1$, then $wij$ is dominated by some vertex in $A_2 \cup A_3$.
 Hence, $g$ is a Roman dominating function on $S(P_n,t)$. Thus, $$\gamma_{_R}(S(P_n,t)) \leq \omega(g)=2\displaystyle \sum_{i=1}^{3}|A_i |+\sum_{i=1}^{3}|C_i |.$$
 On one hand,  $$\displaystyle\sum_{i=1}^{3}|C_i |=n^{t-2}(2|S|+1)=n^{t-2}(2\gamma(P_n)-1)=2k+1$$ and, on the other hand, $$  \vert A_1 \vert=n^{t-2}\displaystyle \sum_{u \in S}u= n^{t-2}\left(\frac{3k^2+k}{2}\right),\, |A_2|=2n^{t-2}$$  and $$|A_3|=n^{t-2}\left( k+2+\displaystyle \sum_{i=2}^{k}3i\right)=n^{t-2}\left(\frac{3k^2+5k-2}{2}\right).$$ 
Thus,  $$\displaystyle \sum_{i=1}^{3}|A_i |=n^{t-2}\left(\frac{3k^2+k}{2}+2+\frac{3k^2+5k-2}{2}\right)=n^{t-2}(3k^2+3k+1).$$ Therefore, $\gamma_{_R}(S(P_n,t))\leq n^{t-2}(6k^2+8k+3)$ and, since $n=3k+2$, $$\gamma_{_R}(S(P_n,t)) \leq n^{t-2}\left( n \left\lceil \frac{2n}{3}\right\rceil  -2 \left\lceil \frac{n}{3}\right\rceil+1\right),$$ as required.\\

\end{proof}

\section{The particular case of cycles}
\begin{theorem} Let  $n \geq 4$ and $t \geq 2$ be two integers. 
If $n\equiv 1,2 \pmod 3$, then $\gamma_{_R}(S(C_n,t))=n^{t-1} \lfloor \frac{2n}{3} \rfloor $, otherwise, $ \frac{n^{t-1}(2n-3)}{3} \le \gamma_{_R}(S(C_n,t))\le  \frac{n^{t-1}(2n-1)}{3}. $
\end{theorem}
\begin{proof} Let $V=\{1,\dots , n\}$ be the vertex set of $C_n$, where $i\in N(i+1)$, for any $i$, and the addition is taken modulo $n$. First, we proceed to deduce the upper bound for $\gamma_{_R}(S(C_n,t))$. 
If $n\equiv 0 \pmod 3$, then Theorem \ref{UpBoundRomanDomSierpinski} leads to  
\begin{equation}
\gamma_{_R}(S(C_n,t)) \leq\frac{n^{t-1}}{3} (2n-1).
\end{equation}
 Suppose that  $n = 3k + 1$, for some integer $k$.  Define $D = \{ ij : i\in V, \ j = i + 1 + 3k', \, 0 \le k' < k - 1 \}$ and $D_{t - 2} = \{ wx : w \in V^{t - 2}, x \in D \}$. Notice that $D$ is a  $2-$packing dominating set, and $D \cap \{ii : i \in V\} = \emptyset $, hence $D_{t - 2}$ is also a $2-$packing dominating set and therefore 
$\gamma(S(C_n, t)) = |D_{t - 2}| = n^{t - 2}  |D| = n^{t-1} \left \lfloor \frac{n}{3} \right \rfloor,$ which implies that
\begin{equation}
 \gamma_{_R}(S(C_n,t))\leq 2 \gamma(S(C_n,t))=n^{t-1} \left\lfloor \frac{2n}{3} \right\rfloor.
\end{equation}
 Now, let $n=3k+2$ for any positive integer $k$. Set $$A=\lbrace wij: \, w \in V^{t-2}, \, i\in V, \, j = i+1+3k' ,\, 0\leq k' \leq k-1 \rbrace $$ and $$B=\lbrace wij: \;  w \in V^{t-2}, \; i\in V, \; j= i-2 \rbrace.$$
 Define $f_2:V^t\rightarrow \lbrace 0,1,2 \rbrace$ such that \\
\begin{eqnarray*}f_2(x)= \left \{ \begin{array}{ll} 

2,& x \in A;
\\
1, &  x \in B;
\\
0, & \text{otherwise.}
\end{array} \right
.\end{eqnarray*}\\
Let $w \in V^{t-2}$ and $i,j\in V$  such that $g(wii)=0$. If $j \equiv i-1\pmod n$, then $wi(i-1) \in N(w(i-1)i) \subset N(A)$. Otherwise, $j \equiv i+3k'$ or $i+2+3k' \pmod n$, for $1 \leq k' \leq k-1 $. Hence, $wij\in N(wi(j+1))$ or $wij\in N(wi(j-1))$ respectively. So $wij \in N(A)$. Therefore, $g$ is a Roman dominating function on $S(C_n,t)$ and, as a consequence,
\begin{equation}
\gamma_{_R}(S(C_n,t))\leq \omega(f_2)=2|A|+|B|=n^{t-1}(2k+1)= n^{t-1}\left \lfloor \frac{2n}{3} \right\rfloor.
\end{equation}
Now we will find the lower bound for $\gamma_{_R}(S(C_n,t))$. Assume that $f=(B_0,B_1,B_2)$ is a $\gamma_{_R}$-function on $S(C_n,t)$. 
Set $$C_{wu}= \lbrace wui \in V_{wu}: \, i\not\in \{u-1,u,u+1 \} \rbrace $$ for $w \in V^{t-2}$ and $u \in V$. 
Hence, induced subgraph on $ C_{wu} $ is isomorphic to $P_{n-3}$ and  $\omega(V_{wu})=\omega(C_{wu})+\displaystyle \sum_{i \in \{ u-1,u,u+1\}} f(wui)$.
 Let $$D_i=\left\{\langle V_{wu}\rangle : \omega(V_{wu})=\left\lceil\frac{2n}{3}\right\rceil-2+i \right\}\text{ for } i\in \{0,1\}$$ and 
 $$D_2=\left\{\langle V_{wu}\rangle : \omega(V_{wu})=\left\lceil\frac{2n}{3}\right\rceil-2+j, \text{ for some } j\ge 2 \right\}.$$
 Notice that  $$
  \gamma_{_R}(S(C_n,t)) 
 =   \sum_{\langle V_{wu}\rangle\in D_0 }\omega(V_{wu})+\sum_{\langle V_{wu}\rangle\in D_1 }\omega(V_{wu})+\sum_{\langle V_{wu}\rangle\in D_2 }\omega(V_{wu}). 
$$ If $\langle V_{wu}\rangle\in D_0 $, then $\{wu(u-1),wuu,wu(u+1) \} \subset B_0$ and so there exists $w' \in V^{t-2}$ and $v \in V$ such that $wuu \in N(w'vv)$ and $f(w'vv)=2$. Thus, $\langle V_{w'v}\rangle\in D_2$. We can define an injective application $\phi: D_0 \longrightarrow D_2$, so that we emphasize that if $\langle V_{wu}\rangle \in D_0$, then  the contribution of $\omega(V_{wu})+\omega(\phi(\langle V_{wu}\rangle))$ to  $\gamma_{_R}(S(C_n,t)) $ is greater than or equal to such contribution when 
both $\langle V_{wu} \rangle$ and $\phi(\langle V_{wu}\rangle )$ are in $D_1$. The argument shows that,  $$\gamma_{_R}(S(C_n,t)) = \displaystyle \sum_{w\in V^{t-2}}  \sum_{u\in V}\omega(V_{wu})\geq n^{t-1} \left(\left\lceil\frac{2n}{3}\right\rceil-1\right).$$
Therefore, the result follows.
 \end{proof}

\section{The particular case of complete graphs}

The domination number of $S(K_n,t)$ was previously studied by Klav\v{z}ar, Milutinovi\'c and Petr in \cite{Klavzar2002} where they obtained the following result.

\begin{theorem}{\rm \cite{Klavzar2002}}\label{DominatingNumberSierpinskiCompletegraphs}
For any integers   $n \geq 2$ and $t \geq 1$, 
\begin{eqnarray*}\gamma (S(K_n,t))=
\left \{ \begin{array}{ll} 

\displaystyle \frac{n^t+n}{n+1},&  t \text{ even};
\\
\\
\displaystyle \frac{n^{t}+1}{n+1}, &  t \text{ odd}.
\end{array} 
\right.
\end{eqnarray*}
\end{theorem}

The above result is an important tool to deduce an upper bound on the Roman domination number of $S(K_n,t)$.

\begin{theorem}
For any integers   $n \geq 2$ and $t \geq 1$, 
\begin{eqnarray*}\gamma_{_R}(S(K_n,t))\le 
\left \{ \begin{array}{ll} 

\displaystyle \frac{2n^t+n-1}{n+1},&  t \text{ even};
\\
\\
\displaystyle \frac{2(n^{t}+1)}{n+1}, &  t \text{ odd}.
\end{array} 
\right.
\end{eqnarray*}
\end{theorem}

\begin{proof}
Let $V=\{1,2,\dots, n\}$ be the vertex set of $K_n$. For $t$ odd we deduce the bound from Theorem \ref{DominatingNumberSierpinskiCompletegraphs}, as $\gamma_{_R}(S(K_n,t))\le 2 \gamma(S(K_n,t))$. We claim that for $t=2k$ there exists a Roman dominating function such that $f(1\dots 1)=1$ and $\omega (f)=\frac{2n^{2k}+n-1}{n+1}$.
To show this we proceed by induction on $k$. For $k=1$ we define the Roman dominating function $f$ as follows. $f(11)=1$, $f(i1)=2$ for all $i\ne 1$ and $f(xy)=0$ for others. Notice that $\omega (f)=2(n-1)+1=\frac{2n^{2}+n-1}{n+1}$.

Now, suppose that $f$ is a Roman dominating function on $S(K_n,2k)$ such that $f(1\dots 1)=1$ and $\omega (f)=\frac{2n^{2k}+n-1}{n+1}$. We shall construct a Roman dominating function $f'$ on $S(K_n,2k+2)$ in the following way: 
\begin{itemize}
\item   $f'(11w)=f(w)$ for all $w\in V^{2k}$. 
\item  $f'(1i\dots i)=0$ for all $i\ne 1$ and $f(11w)=f(w')$ for all $w\in V^{2k-2}\setminus\{i\dots i: i\in V\}$, where $w'$ is obtained from $w$ by exchanging $i$ and $1$.
\item For any $i\in V\setminus\{1\}$ and $w\in V^{2k}$, we define $f(i1w)$ as follows. As shown in \cite[Corollary 3.5]{Klavzar2002}, there exists a $1$-perfect code $C$ of $S(K_n,2k)$ which contains all the extreme vertices. So, we set $f'(i1w)=2$ for all $w\in C$ and $f'(i1w)=0$ for others. 
\item $f'(ij1\dots 1)=0$ and $f'(ijw)=f(w)$ for all $i,j\ne 1$ and $w\ne 1\dots 1$.
\end{itemize}
Notice that $f'(1\dots 1)=1$. To conclude that $f'$ is a Roman dominating function on $S(K_n,2k+2)$ we only need to observe that all $x\in V^{2k+2}$ of the form $x=1i\dots i$, $i\ne 1$ are adjacent to $i1\dots 1$   and $f'(i1\dots 1)=2$, and all $x\in V^{2k+2}$ of the form $x=ij1\dots 1$, $i,j\ne 1$ are adjacent to $i1j\dots j$   and $f'(i1j\dots j)=2$. 
Finally, by Theorem \ref{DominatingNumberSierpinskiCompletegraphs}, $|C|=\frac{n^{2k}+n}{n+1}$, and so
$$
\omega(f')=\omega(f)+(n-1)(\omega(f)-1)+2|C|(n-1)+(n-1)^2(\omega(f)-1)
               =\frac{2n^{2k+2}+n-1}{n+1},
$$
as required. 
\end{proof}

By Remark \ref{RomanSubgraph} we deduce the following corollary. 

\begin{corollary}
For any graph $G$ of order $n$ and any integer $t$,  
$$\gamma_{_R}(S(G,t)) \geq  \gamma_{_R}(S(K_n,t)) .$$
\end{corollary}

As the above corollary shows, a lower bound (or a closed formula) on the Roman domination number of $S(K_n,t)$ imposes a lower bound on $\gamma_{_R}(S(G,t))$ for every graph $G$. Therefore, this issue definitely deserves further research.
 
\bibliographystyle{elsart-num-sort}

\end{document}